\documentclass[12pt]{amsart}
\usepackage{amsmath,amsthm,amssymb}

\newcommand{\al}{\alpha}
\newcommand{\be}{\beta}

\newcommand{\fr}{\mathcal{F}}
\newcommand{\ity}{\infty}
\newcommand{\C}{\mathbb{C}}

\numberwithin{equation}{section}
\newtheorem{theorem}{Theorem}[section]
\newtheorem{lemma}[theorem]{Lemma}

\theoremstyle{remark}

\makeatletter
\@namedef{subjclassname@2010}{%
  \textup{2010} Mathematics Subject Classification}
\makeatother

\begin{document}

\title[On a Theorem of  Schwick ]{On a Theorem of  Schwick }

\thanks{The research work of the first author is supported by research fellowship from UGC India.}

\author[G. Datt]{Gopal Datt}
\address{Department of Mathematics, University of Delhi,
Delhi--110 007, India} \email{ggopal.datt@gmail.com }

\author[S. Kumar]{Sanjay Kumar}

\address{Department of Mathematics, Deen Dayal Upadhyaya College, University of Delhi,
Delhi--110 015, India }
\email{sanjpant@gmail.com}

\begin{abstract}
Let $\mathcal D $ be a domain, $n, k$ be positive integers and $n\geq k+3$. Let $\mathcal F$ be a family of functions meromorphic  in $\mathcal D$.  If each $f\in \mathcal F$ satisfies $(f^n)^{(k)}(z)\neq 1$ for $z\in \mathcal D$, then $\mathcal F$ is a normal family. This  result was proved by Schwick ~\cite{Sch1}, in this paper we extend this theorem.

\end{abstract}



 \maketitle

\section{Introduction and main results}
We denote the  complex plane by $\C$, and the unit disk $\{z\in \C: \ |z|<1 \}$ by $\Delta$.
In 1989, Schwick ~\cite{Sch1} proved a normality criterion which states that: {\it{For positive integers $k$,  $n\geq k+3$, let $\mathcal F$ be a family of functions meromorphic  in $\mathcal D$.  If each $f\in \mathcal F$ satisfies $(f^n)^{(k)}(z)\neq 1$ for $z\in \mathcal D$, then $\mathcal F$ is a normal family.}} This result holds good for holomorphic functions with the case $n\geq k+1$. The following theorem is a result of Wang and Fang ~\cite{wang}. The proof was omitted in that article, here we give a proof of this result and extend this theorem.
\begin{theorem}\label{thm}
Let $n, k$ be positive integers and $n\geq k+1$ and $\mathcal D $ be a domain in $\C$. Let $\mathcal F$ be a family of functions meromorphic  on $\mathcal D$.  If each $f\in \mathcal F$ satisfies $(f^n)^{(k)}(z)\neq 1$ for $z\in \mathcal D$, then $\mathcal F$ is a normal family.\\
\end{theorem}

It is natural to ask what can happen if we have a solution of $(f^n)^{(k)}-1$. For this question we can extend Theorem \ref{thm} for the case $k\geq 1$ in the following manner.

\begin{theorem}\label{thm1}
Let $n, k$ be positive integers and $n\geq k+2$ and $\mathcal D $ be a domain in $\C$. Let $\mathcal F$ be a family of functions meromorphic  on $\mathcal D$.  If for each function $f\in \fr, \ (f^n)^{(k)}(z)- 1$ has at most one zero ignoring multiplicity $(IM )$ in  $\mathcal D$, then $\mathcal F$ is a normal family.\\

\end{theorem}
 In this paper, we use the following standard notations of value distribution theory,
\begin{center}
$T(r,f); m(r,f); N(r,f); \overline{N}(r,f),\ldots$.
\end{center}
We denote $S(r,f)$ any function satisfying
\begin{center}
$S(r,f)=o\{T(r,f)\}$,  as $r\rightarrow +\ity,$
\end{center}
 possibly outside of a set with finite measure.\\

 \section{Preliminary results}
        In order to prove our results we need  the following Lemmas.\\

\begin{lemma}\{~\cite{Zalc}, p. 216; ~\cite{Zalc}, p. 814\}\label{lem1}(Zalcman's lemma)\\Let $\mathcal F$ be a family of meromorphic  functions in the unit disk  $\Delta$, with the property that for every function $f\in \mathcal F,$  the zeros of $f$ are of multiplicity at least $l$ and the poles of $f$ are of multiplicity at least $k$ . If $\mathcal F$ is not normal at $z_0$ in $\Delta$, then for  $-l< \alpha <k$, there exist
\begin{enumerate}
\item{ a sequence of complex numbers $z_n \rightarrow z_0$, $|z_n|<r<1$},
\item{ a sequence of functions $f_n\in \mathcal F$ },
\item{ a sequence of positive numbers $\rho_n \rightarrow 0$},
\end{enumerate}
such that $g_n(\zeta)=\rho_n^{\alpha}f_n(z_n+\rho_n\zeta) $ converges to a non-constant meromorphic function $g$ on $\C$ with $g^{\#}(\zeta)\leq g^{\#}(0)=1$. Moreover $g$ is of order at most two . \\
\end{lemma}

\begin{lemma}\label{lem2}\{~\cite{lahiri}, Lemma 2.5, Lemma 2.6; \ ~\cite{lahiri 1}, Lemma 2.2\}
Let $R=\frac{P}{Q} $ be a rational function and $Q$ be non-constant. Then $(R^{(k)})_{\ity}\leq(R)_{\ity}-k,$ where $k$ is a positive integer,  $(R)_{\ity}=$ $\text{deg}(P)-\text{deg}(Q)$ and $\text{deg}(P)$ denotes the degree of P.\\
\end{lemma}

\begin{lemma}\label{lem3}\ {~\cite{lahiri}}
Let $R = a_mz^m + \ldots + a_1z + a_0 +\frac{P}{B},$ where $a_0, a_1, \ldots, a_{m-1}, a_m(\neq0) $ are constants, $m$ is a positive integer and $P,\  B$ are polynomials with deg$(P)<$ deg$(B)$. If $k\leq m$,  then $(R^{(k)})_{\ity} = (R)_{\ity}-k,$ \\
\end{lemma}

\begin{lemma}\label{lem4}
Let $k, \ n$ is  two positive integer and  $n\geq k+1.$ Let $f$ be a non-constant rational function then $(f^n)^{(k)}-b$ has a root for all nonzero complex numbers $b$.\\
\end{lemma}
\begin{proof}
Suppose $(f^n)^{(k)}- b$ has no root. First we suppose $f$ is a non-constant polynomial of degree $d \geq 1$, then $(f^n)^{(k)}-b$ is a polynomial of degree $nd-k \geq 1$. Thus by fundamental theorem of  algebra $(f^n)^{(k)}-b$ has a solution.\\


Again, let $f$ is a non-polynomial  rational function. We set
 \begin{equation}\label{eq1}
 f(z)= A\frac{(z-\al_1)^{m_1}(z-\al_2)^{m_2}\ldots (z-\al_s)^{m_s}}{(z-\be_1)^{n_1}(z-\be_2)^{n_2}\ldots (z-\be_t)^{n_t}},
 \end{equation}
 where $A$ is a nonzero constant and $m_1, m_2, \ldots, m_s,n_1, n_2, \ldots, n_t$ are positive integers. We denote
 \begin{equation}\notag
 M=n\sum_{i=1}^{s}m_i, \ N=n\sum_{i=1}^{t}n_i.
 \end{equation}
 \begin{equation}\label{eq2}
 (f^n)^{(k)}(z) = \frac{(z-\al_1)^{nm_1-k}(z-\al_2)^{nm_2-k}\ldots (z-\al_s)^{nm_s-k}g(z)}{(z-\be_1)^{nn_1+k}(z-\be_2)^{nn_2+k}\ldots (z-\be_t)^{nn_t+k}}=\frac{p}{q},
 \end{equation}
 where $g(z)$ is a polynomial and deg$(g)\leq k(s+t-1).$ Suppose $(f^n)^{(k)}(z)\neq b$, then
 \begin{equation}\label{eq3}
 (f^n)^{(k)}(z)= b + \frac{B}{(z-\be_1)^{nn_1+k}(z-\be_2)^{nn_2+k}\ldots (z-\be_t)^{nn_t+k}}=\frac{p}{q}
 \end{equation}
 from\eqref{eq2} and \eqref{eq3} $N+kt=$ deg$(q)=$ deg$(p)=M-ks+$ deg$(g)\leq M-ks+k(s+t-1)=M+kt-k $. This gives $M-N\geq k$ $i.e. n(\sum^{i=1}_{s}m_i-\sum^{i=1}_{t}n_i)\geq k$. This implies $(\sum_{i=1}^{s}m_i-\sum_{i=1}^{t}n_i)> 1$. So $(f)_{\infty}>1$ hence $(f^n)_{\infty}>n$. Therefore we can express $f^n $ as follows
 \begin{equation}\notag
 f^n(z)= a_mz^m + \ldots + a_1z + a_0 +\frac{P}{B},
 \end{equation}
  where $a_0, a_1, \ldots, a_{m-1}, a_m(\neq0) $ are constants, $m\geq n$ is an integer, $P$ and $ B$ are polynomials with deg$(P)<$ deg$(B)$. Since $m>k$, then by \ref{lem3} we get
  \begin{equation}\notag
  ((f^n)^{(k)})_{\ity}=(f^n)_{\ity}-k > n-k\geq 1,
  \end{equation}
which contradicts the fact that deg$(p) =$  deg$(q)$.  Thus $(f^n)^{(k)}(z)- b$ has a solution in $\C.$\\
 \end{proof}
 \
\begin{lemma}\label{lemmali}~\cite{Li}
Let $n, k$ be positive integers such that $n\geq k+2$ and $a\neq0$ be a finite complex number, and $f$ be a non-constant rational meromorphic function, then $(f^n)^{(k)} - a$ has at least two distinct zeros.
\end{lemma}
 \begin{lemma}\label{lemm1}\{~\cite{Yang} P.38\}
 Let $f(z)$ be a transcendental  meromorphic function on $\C$ , then
 \begin{equation}\notag
 m(r, \frac{f^{(k)}}{f}) = S(r, f)
 \end{equation}
 for every positive integer $k$.\\
\end{lemma}

\begin{lemma}\label{lemm2}
Let $f$ be a non-constant meromorphic function on $\C$. Then
\begin{equation}\notag
T(r, f^{(k)} \leq T(r, f) + k \overline{N}(r, f) + S(r, f)
\end{equation}
\end{lemma}
\begin{proof}
\begin{align*}
T(r, f^{(k)}) & = N(r, f^{(k)}) + m(r, f^{(k)})\\
&\leq  N(r, f) + k\overline{N}(r, f) + m(r, f) + m(r, \frac{f^{(k)}}{f})\\
& \leq T(r,f) + k\overline{N}(r, f) + S(r, f)
\end{align*}
\end{proof}
\begin{lemma}\label{lem5}~\cite{frank}
Let $f(z)$ be a transcendental meromorphic function. Then for each positive number $\epsilon$ and each positive integer $k$, we have
\begin{equation}\notag
k\overline{N}(r,f)\leq N(r, \frac{1}{f^{(k)}}) + N(r,f) + \epsilon T(r,f) +S(r,f).\\
\end{equation}
\end{lemma}

\begin{lemma}\label{lem6}\{~\cite{berg} Corollary 3.\}
 If a meromorphic function of finite order $\rho$ has only finitely many critical values, then it has at most $2\rho $ asymptotic values.\\
\end{lemma}

\begin{lemma}\label{lem7}~\cite{berg1}
Let $g(z)$ be a transcendental meromorphic function and suppose that  $g(0)\neq \ity$ and the set of finite critical and asymptotic values of $g(z)$ is bounded. then there exists $R>0$ such that
\begin{equation}\notag
|g'(z)|\geq \frac{|g(z)|}{2\pi|z|}\log\frac{|g(z)|}{R},
\end{equation}
for all $z\in \C\setminus\{0\}$ which are not poles of $g(z)$.\\
\end{lemma}

\begin{lemma}\label{lem8}~\cite{wang1}
If $f$ is a trancendental meromorphic function and $k$ be a positive integer, then, for every positive number $\epsilon$,
\begin{equation}\notag
(k-2)\overline{N}(r, f) + N(r, \frac{1}{f})\leq 2 \overline{N}(r,\frac{1}{f})+ N(r, \frac{1}{f^{(k)}}) + \epsilon T(r, f) + S(r,f).\\
\end{equation}
\end{lemma}

The following lemma was proved by Bergweiler~\cite{berg} and Wang~\cite{wang} independently. Here we are giving another proof of this lemma.
\begin{lemma}\label{lemma}
Let $f(z)$ be a transcendental meromorphic function with finite order. Let $k, n$ be two positive integers such that $n\geq k+1$, then $(f^n)^{(k)}- b $ has infinitely many zeros for all $b\in\C\setminus \{0\}$.
\end{lemma}
\begin{proof}
Suppose on the contrary that $(f^n)^{(k)}$ assumes the value $b$ only finitely many times. Then
\begin{equation}\label{eq5}
N(r, \frac{1}{(f^n)^{(k)}-b}) = O(\log r)= S(r, f).
\end{equation}
By Nevanlinna's First Fundamental Theorem and Lemma \ref{lemm1} and Lemma \ref{lemm2}
\begin{align}
m(r, \frac{1}{f^n}) + &  m(r, \frac{1}{(f^n)^{(k)}-b})\notag\\
&\leq  m(r, \frac{(f^n)^{(k)}}{f^n}) + m(r, \frac{1}{(f^n)^{(k)}}) + m(r, \frac{1}{(f^n)^{(k)}-b})\notag\\
& \leq  m(r, \frac{1}{(f^n)^{(k)}} + \frac{1}{(f^n)^{(k)}-b}) + S(r, f^n)\notag\\
& \leq m(r, \frac{1}{(f^n)^{(k+1)}}) + m (r, \frac{(f^n)^{(k+1)}}{(f^n)^{(k)}} + \frac{(f^n)^{(k+1)}}{(f^n)^{(k)}-b}) + S(r, f^n)\notag\\
&\leq m(r, \frac{1}{(f^n)^{(k+1)}}) + S(r, f^n)\notag\\
&\leq T(r, (f^n)^{(k+1)}) - N(r, \frac{1}{(f^n)^{(k+1)}} + S(r, f)\notag\\
&\leq T(r, (f^n)^{(k)}) + \overline{N}(r, f^n) - N(r, \frac{1}{(f^n)^{(k+1)}}) + S(r, f^n).\label{eq6}
\end{align}
Together with Nevanlinna's First Fundamental Theorem this yields
\begin{align}
T(r, f^n)& \leq \overline{N}(r, f^n )  + N(r, \frac{1}{f^n} ) \notag \\
& + N(r, \frac{1}{(f^n)^{(k)}-b}) - N(r, \frac{1}{(f^n)^{(k+1)}}) + S(r, f^n). \label{eq7}
\end{align}
First, we consider the case  when $k\geq 2$, then By Lemma \ref{lem8}, for every $\epsilon >0$, we have
 \begin{align}
\overline{N}(r, f^n) &+ N (r, \frac{1}{f^n})\notag\\
& \leq (k-1)\overline{N}(r, \frac{1}{f^n}) + N(r, \frac{1}{f^n})\notag\\
&\leq 2\overline{N}(r, \frac{1}{f^n}) + N(r, \frac{1}{(f^n)^{(k+1)}}) + \epsilon T(r, f^n) + S(r, f^n). \label{eq8}
\end{align}
From \eqref{eq7} and \eqref{eq8}, and using the fact that zeros of $f^n$ has multiplicity at least 3 in this case, we get
\begin{align}
T(r, f^n)&\leq 2\overline{N}(r, \frac{1}{f^n}) +  N(r, \frac{1}{(f^n)^{(k)}-b}) + \epsilon T(r, f^n) + S(r, f^n)\notag\\
&\leq \frac{2}{3}N(r, \frac{1}{f^n}) + N(r, \frac{1}{(f^n)^{(k)}-b}) + \epsilon T(r, f^n) + S(r, f^n)\notag\\
&\leq \frac{2}{3} T(r, \frac{1}{f^n}) +  N(r, \frac{1}{(f^n)^{(k)}-b}) + \epsilon T(r, f^n) + S(r, f^n)\notag\\
&\leq (\frac{2}{3} + \epsilon)T(r, f^n) +  N(r, \frac{1}{(f^n)^{(k)}-b}) + S(r, f^n).\label{eq9}
\end{align}
Now, taking $\epsilon =\frac{1}{6}$, from \eqref{eq5} and \eqref{eq9},  we obtain
\begin{equation}\notag
T(r, f^n ) \leq 6 N(r, \frac{1}{(f^n)^{(k)}-b}) + S(r, f^n) = S(r, f^n),
\end{equation}
 which contradicts the fact that $f$ is a transcendental meromorphic function. Thus, Lemma \ref{lem8} is proved for the case $k\geq2$.\\

 Now, for the case $k=1$, we use the method of Fang ~\cite{Fang}.  We first consider that $f(z)$ has only finitely many zeros so is $f^n(z)$ has only finitely many zeros $i.e. N(r, \frac{1}{f^n}) = S(r, f^n) $. and invoke the Lemma \ref{lem5} and combine it with \eqref{eq7}, we have
 \begin{align}
 T(r, f^n) & \leq \overline{N}(r, f^n) + N(r, \frac{1}{f^n})\notag\\
 & + N(r, \frac{1}{(f^n)'-b})- N(r, \frac{1}{(f^n)''}) + S(r, f^n)\notag\\
 & \leq\frac{1}{2} N(r, f^n) + N (r, \frac{1}{f^n}) + N(r, \frac{1}{(f^n)'-b})\notag\\
 & + \frac{1}{4} T(r, f^n) + S(r, f^n)\notag\\
 & \leq \frac{3}{4} T(r, f^n) + N(r, \frac{1}{(f^n)'-b}) + S(r, f^n) \notag
  \end{align}
  Thus, we obtain
  \begin{equation}\notag
  T(r, f^n) = S (r, f^n).
  \end{equation}
Which is a contradiction, therefore the theorem is valid in this case.
Now, consider the case when $f(z)$ has infinitely many zeros $\{z_i\}, i=1, 2, 3, \ldots$. Define
\begin{center}
$g(z) = f^n(z) -bz ,$ then \
$g'(z) = (f^n)'(z) - b.$
\end{center}
If we show that $g'(z)$ has infinitely many zeros then we have done. Suppose $g'(z)$ has only finitely many zeros, so  $g(z)$ has only finitely many critical values and hence $g(z)$ has only finitely many asymptotic values. Without any loss of generality we may assume that $f(0)\neq \ity , $  thus by Lemma \ref{lem7}, we get
\begin{equation}\notag
|g'(z_i)|\geq \frac{|g(z_i)|}{2\pi|z_i|}\log\frac{|g(z_i)|}{R},
\end{equation}
this shows
\begin{equation}\notag
\frac{|z_ig'(z_i)|}{|g(z_i)|}\geq \frac{1}{2\pi}\log\frac{|g(z_i)|}{R},
\end{equation}
Since $\frac{1}{2\pi}\log\frac{|g(z_i)|}{R}\rightarrow \ity$ as $i\rightarrow \ity$, $\frac{|z_ig'(z_i)|}{|g(z_i)|}\rightarrow \ity$ as $i\rightarrow \ity.$ But $\frac{|z_ig'(z_i)|}{|g(z_i)|}\rightarrow 1$ as $i\rightarrow \ity,$ a contradiction. Hence we deduce that $(f^n)'(z)-b$ has infinitely many zeros. This completes the proof of theorem.\\

\end{proof}

\begin{lemma}\label{lem9} ~\cite{clunie}
 Let $f$ be an entire function. If the spherical derivative $f^{\#}$ is  bounded in $\C$, then the order of $f$ is at most 1.
\end{lemma}

\section{Proof of Theorem \ref{thm}}

Since normality is a local property, we assume that $D=\Delta=\{z:|z|<1\}$.
Suppose  $\fr$ is not normal in $D$. Without loss of generality we assume that $\fr$ is not normal at the point $z_0$ in $\Delta$.  Then by Lemma \ref{lem1}, there exist
\begin{enumerate}
\item { a sequence of complex numbers $z_j \rightarrow z_0$, $|z_j|<r<1$},
\item { a sequence of functions $f_j\in \mathcal F$ and}
\item { a sequence of positive numbers $\rho_j \rightarrow 0$,}
\end{enumerate}
such that $g_j(\zeta) = \rho_{j}^{-\frac{k}{n}}f_j(z_j + \rho_j\zeta)\rightarrow g(\zeta)$ converges locally uniformly to a non-constant meromorphic  function $g(\zeta)$ in $\C$ with $g^{\#}(\zeta)\leq g(0)=1$. Moreover $g$ is of order at most two. We see that
\begin{equation}\label{eq3.1}
(g_{j}^{n})^{(k)}(\zeta)= (f_{j}^{n})^{(k)}(z_j + \rho_j\zeta) \rightarrow (g^n)^{(k)}(\zeta)
\end{equation}
converges locally uniformly with respect to the spherical metric. By Hurwitz's Theorem, $(g^n)^{(k)}\equiv 1$ or $(g^n)^{(k)}\neq 1.$ \\

 Let $(g^n)^{(k)}\equiv 1$, Then $g$ has no pole  this implies that $g$ is an entire function having no zero. Since $g^{\#}\leq 1$, we may put $g(\zeta)= \exp{(c\zeta +d)}$, where $c(\neq 0)$ and $d$ are constants. therefore we get
 \begin{equation}\notag
 (nc)^k\exp{(c\zeta + d)}\equiv 1,
 \end{equation}
which is not possible.\\

  Thus $(g^n)^{(k)}\neq 1$, which contradicts Lemma \ref{lem4} and Lemma \ref{lemma}. Thus $\fr $ is normal in $\mathcal D$. This completes the proof of theorem.\\

  \section{Proof of Theorem \ref{thm1}}

Since normality is a local property, we assume that $D=\Delta=\{z:|z|<1\}$.
Suppose  $\fr$ is not normal in $D$. Without loss of generality we assume that $\fr$ is not normal at the point $z_0$ in $\Delta$.  Then by Lemma \ref{lem1}, there exist
\begin{enumerate}
\item { a sequence of complex numbers $z_j \rightarrow z_0$, $|z_j|<r<1$},
\item { a sequence of functions $f_j\in \mathcal F$ and}
\item { a sequence of positive numbers $\rho_j \rightarrow 0$,}
\end{enumerate}
such that $g_j(\zeta) = \rho_{j}^{-\frac{k}{n}}f_j(z_j + \rho_j\zeta)\rightarrow g(\zeta)$ converges locally uniformly to a non-constant meromorphic  function $g(\zeta)$ in $\C$ with $g^{\#}(\zeta)\leq g(0)=1$. Moreover $g$ is of order at most two. We see that
\begin{equation}\label{eq4.1}
(g_{j}^{n})^{(k)}(\zeta)= (f_{j}^{n})^{(k)}(z_j + \rho_j\zeta) \rightarrow (g^n)^{(k)}(\zeta)
\end{equation}
converges locally uniformly with respect to the spherical metric.\\

Now we claim $(g_{j}^{n})^{(k)}-1$ has at most one zero IM. Suppose $(g_{j}^{n})^{(k)}-1$ has two distinct zeros $\zeta_0$ and $\zeta^*_0$ and choose $\delta > 0$ small enough so that $D(\zeta_0, \delta)\cap D(\zeta^*_0, \delta) = \emptyset$ and $(g_{j}^{n})^{(k)}-1$ has no other zeros in $D(\zeta_0,\delta)\cup D(\zeta_{0}^*,\delta)$, where $D(\zeta_0,\delta) = \{\zeta : |\zeta-\zeta_0|<\delta\}$ and $D(\zeta_{0}^*,\delta) = \{\zeta : |\zeta-\zeta_{0}^*|<\delta\}$. By Hurwitz's theorem, there exist two sequences $\{\zeta_{j}\}\subset D(\zeta_0,\delta), \{\zeta_{j}^*\}\subset D(\zeta_{0}^*,\delta)$ converging to $\zeta_{0}, \ \text{and}\  \zeta_{0}^*$ respectively and from \eqref{eq4.1}, for sufficiently large j, we have\\
\begin{equation}\notag
(f_j^n)^{(k)}(z_j+ \rho_j \zeta_j)- 1 = 0 \ \text{and} \ (f_j^n)^{(k)}(z_j+ \rho_j \zeta^*_j)- 1 = 0.
\end{equation}
Since $z_j\rightarrow 0$ and $\rho_j\rightarrow 0$, we have $z_j+\rho_j\zeta_j\in D(\zeta_0,\delta)$ and $z_j+\rho_j\zeta_{j}^*\in D(\zeta_{0}^*,\delta)$
 for sufficiently large $j$, so  $(f_j^n)^{(k)}- 1 $ has two distinct zeros, which contradicts the fact that $(f_j^n)^{(k)}- 1$ has at most one zero. But Lemma \ref{lemmali} and Lemma \ref{lemma} confirms the non existence of such non-constant  meromorphic function. This contradiction shows that $\fr$ is normal in $\Delta$ and this proves the theorem.\\

\end{document}